\documentclass[11pt]{amsart}
\usepackage{amsfonts}
\usepackage{bbm}
\usepackage{amsfonts,amssymb,amsmath,amsthm}
\usepackage{url}
\usepackage{enumerate}
\usepackage{bbm}
\usepackage[all]{xy}
\usepackage[pdftex, colorlinks, citecolor=red, backref=page]{hyperref}

\newtheorem{thrm}{Theorem}[section]
\newtheorem{lem}[thrm]{Lemma}
\newtheorem{prop}[thrm]{Proposition}
\newtheorem{cor}[thrm]{Corollary}
\theoremstyle{definition}
\newtheorem{definition}[thrm]{Definition}
\newtheorem{remark}[thrm]{Remark}
\numberwithin{equation}{section}

\author{Qingnan An}
\address{Postdoctoral Research Station of Mathematics, Hebei Normal University, Shijiazhuang, Hebei, China\ 050016}
\email{369824037@qq.com}
\author{George A. Elliott}
\address{Department of Mathematics, University of Toronto, Toronto, Ontario, Canada \ M5S 2E4}
\email{elliott@math.toronto.edu}
\author{Zhiqiang Li}
\address{College of Mathematics and Statistics, Chongqing University, Chongqing, China \ 401331}
\email{zqli@cqu.edu.cn}
\author{Zhichao Liu}
\address{Postdoctoral Research Station of Mathematics, Hebei Normal University, Shijiazhuang, Hebei, China\ 050016}
\email{lzc.12@outlook.com}


\begin{document}

\title[On the classification of C*-algebras of real rank zero, \emph{V}]{On the classification of C*-algebras of real rank zero, \emph{V}}

\begin{abstract}
In this paper, we continue to work under the scheme of the Elliott program in the setting of real rank zero C*-algebras. In particular, using ordered total K-theory, we give a classification for the real
rank zero inductive limits of direct sums of generalized dimension drop interval algebras.

\end{abstract}

\maketitle
\section{introduction}
A large part of the history of the classification of amenable C*-algebras, beginning with the UHF algebras of Glimm (\cite{Gli}, see Bratteli \cite{Bra}, also Dixmier \cite{Dix}), and the AF algebras of Bratteli (\cite{Bra}, see also Elliott \cite{Ell1}),  has been the consideration of the question in the case of real rank zero. (Being of real rank zero means density of the invertible elements in the set of all self-adjoint elements.) This was the setting for both the classification of the finite AT algebras considered in \cite{Ell}, and that of the infinite algebras (limits of direct sums of Cuntz algebras) considered by Rordam in \cite{Ror}. (In the first case real rank zero is a restriction, while in the second case it was automatic.) It was also the setting for the much more general finite, respectively, infinite algebras considered in \cite{EG} and \cite{DG}, on the one hand, and in \cite{BEEK} and \cite{KP} (also in \cite
{Phi}) on the other hand.

Interestingly, while classification results for algebras not of real rank zero definitely require additional restrictions, e.g., in the simple case to Jiang-Su stable algebras, in the real rank zero case, such restrictions are redundant.

In a number of articles, \cite{Ell}, \cite{Ei}, and \cite{DG}, particular attention has been given to the case of limits of dimension drop interval algebras. (The case of ordinary interval---no dimension drops---is easy since a real rank zero limit of such algebras must in fact be an AF algebra.) These papers consider the case of what might be known as classical dimension drop interval algebras---with the same dimension drop at both ends of the interval. More recent papers---\cite{Th}, \cite{My}, \cite{JS}, and \cite{Li}---go further, establishing classification in the simple case for inductive limits of (finite direct sums of) general dimension drop interval algebras (possibly different dimension drops at the two endpoints). Since, as suggested in \cite{GL} and \cite{AE}, some subtleties arise in the search for an existence theorem that might be useful in an intertwining argument for general non-simple limits (of the more general non-classical dimension drop algebras), it seems appropriate to show that, at least in the real rank zero setting (and so, perhaps, also in the case with the ideal property), a suitable existence theorem can in fact be formulated (in the general dimension drop case), which together with an appropriate uniqueness theorem, makes an intertwining argument to prove isomorphism possible.

 %  Briefly, the difficulty illustrated in \cite{GL} is that KK-elements that are positive in the Dadarlat-Loring sense do not come from algebra homomorphisms, necessarily----under the $K_0$-multiplicity is large. In the real rank zero setting, however, it is possible to reduce to the case that the $K_0$-multiplicity is either large or zero in which case the (possible) KK-element is also zero.

   %The present result in the case of classical dimension drop interval algebras is included in the main result of \cite{DG}. It seems worthwhile to check that the classification there (in the simple real rank zero case) still holds in the non-simple case.
   %We still do not know if the range of the invariant is the same in the two cases----i.e., if the classes are coextamise.

The real rank zero inductive limits we will focus on involve sequences as follows: $$\bigoplus\limits_{i}^{}\textrm{M}_{l_i}(\mathrm{I}[m_{0,i},m_i,m_{1,i}])
\stackrel{\phi_1}\rightarrow\bigoplus\limits_{j}^{}\textrm{M}_{k_j}
(\mathrm{I}[n_{0,j},n_j,n_{1,j}])\stackrel{\phi_2}\rightarrow...\rightarrow \textrm{A},$$ where $\mathrm{I}[m_0, m, m_1]$ is a generalized dimension drop interval algebra (see Definition \ref{JSD}), and all direct sums are finite. The main result of this paper is the following classification theorem.
\begin{thrm}\label{cla}
 Let $A=\underrightarrow{\lim}(A_n,\phi_{n,m})$ and $B=\underrightarrow{\lim}(B_n,\psi_{n,m})$ be
 two real rank zero inductive limit C*-algebras of direct sums of generalized dimension drop
 interval algebras. Suppose that there is an isomorphism of ordered scaled groups
 $$
 \alpha: (\underline{\mathrm{K}}(A), \underline{\mathrm{K}}^+(A), \Sigma(A))\rightarrow(\underline{\mathrm{K}}(B), \underline{\mathrm{K}}^+(B), \Sigma(B))
 $$ which preserves the action of the Bockstein operations, then there is an isomorphism $\phi: A\rightarrow B$ with $\phi_*=\alpha$. (For ordered total K-theory, see Definitions 2.2--2.5.)

\end{thrm}
This paper is organized as follows. In Section 2, we collect some preliminary concepts concerning ordered total K-theory and generalized dimension drop interval algebras. In Section 3, we develop some technical results (concerning the weak variation and decomposition) for real rank zero inductive limits. In Section 4, we establish existence results, including the main existence theorem we use in the intertwining argument. In Section 5, several (stable) uniqueness theorems are proved, and Theorem \ref{cla} is confirmed.

\section{notation and preliminaries}
In this section, for the convenience of readers, we collect some necessary definitions and set up notation.

\begin{definition}[\cite{Ell}]
The classical dimension drop interval algebra refers to the C*-algebra
$$
I_p=\{f\in \mathrm{M}_p(C_0(0,1]):\,f(1)=\lambda\cdot1_p,\,1_p {\rm \,is\, the\, identity\, of}\, \mathrm{M}_p\},
$$
and the C*-algebra $\widetilde{I}_p$ obtained by adjoining a unit to $I_p$.
\end{definition}

The invariant we concern is the K-theory with coefficients, see \cite{Go}, \cite{DL1} and \cite{DG}.
\begin{definition}[\cite{DL1}, \cite{DG}] For a natural number $p$ (not necessarily prime) and a C*-algebra $A$, the K-theory of $A$ with coefficients in $\mathbb{Z}_p\,(=\mathbb{Z}/p\mathbb{Z})$ can be formulated as follows: $\mathrm{K}_*(A;\mathbb{Z}_p)=\mathrm{KK}(I_{p},A\otimes C(S^1))$ for $p\geq1$, and $\mathrm{K}_*(A;\mathbb{Z}_p)=\mathrm{KK}(\mathbb{C}, A\otimes C(S^1))=\mathrm{K}_*(A)$ for $p=0$. Note that $\mathrm{K}_*(A;\mathbb{Z}_p)=0$ for $p=1$.

Let us set
$$\mathrm{K}_*(A;\mathbb{Z}\oplus\mathbb{Z}_p)=\mathrm{K}_*(A)\oplus \mathrm{K}_*(A;\mathbb{Z}_p),$$
or equivalently,
$$\mathrm{K}_*(A;\mathbb{Z}\oplus\mathbb{Z}_p)\cong \mathrm{KK}(\widetilde{I_{p}},A\otimes C(S^1)),$$
for all $p\geq0$, with $I_0=\mathbb{C}$.
\end{definition}

\begin{definition}[\cite{DL1}, \cite{DG}]
For a  C*-algebra $A$, the total $\mathrm{K}$-theory of $A$ is defined as
$$
\underline{\mathrm{K}}(A)=\bigoplus_{p=0}^\infty \mathrm{K}_* (A;\mathbb{Z}_p).
$$
\end{definition}

\begin{definition}[\cite{DL1}, \cite{DG}] (Dadarlat-Loring order.) The order structure we shall work with is $\mathrm{K}_*(A;\mathbb{Z}\oplus\mathbb{Z}_p)^{+}$, which can be identified as the image of the abelian semigroup
$[\widetilde{I_{p}},A\otimes C(S^1)\otimes \mathcal{K}]$ in $\mathrm{KK}(\widetilde{I_{p}},A\otimes C(S^1))\,(\cong \mathrm{K}_*(A;\mathbb{Z}\oplus\mathbb{Z}_p))$.
\end{definition}

%\begin{definition}[\cite{DL1}, \cite{DG}]
%For a C*-algebra $A$, the invariant we are concerned with is the tuple $$(\underline{\mathrm{K}}(A),\underline{\mathrm{K}}^+(A),\Sigma(A)),$$
%where $\underline{\mathrm{K}}^+(A)$ is the cone generated by $\mathrm{K}_*(A;\mathbb{Z}\oplus\mathbb{Z}_p)^{+}$, and
%$\Sigma(A)$ is the scale of $A$.
%\end{definition}

%We shall say
%$$(\underline{\mathrm{K}}(A),\underline{\mathrm{K}}^+(A),\Sigma(A))
%\cong
%(\underline{\mathrm{K}}(B),\underline{\mathrm{K}}^+(B),\Sigma(B)),$$
%if there is an ordered graded isomorphism $\rho:\,{\underline{\mathrm{K}}}(A)\to{\underline{\mathrm{K}}}(B)$, which preserves the scale and also preserves the action of the Bockstein operations.

\begin{definition}
  Let $A$ and $B$ be unital C*-algebras. Two homomorphisms $\phi,\psi:A\rightarrow B$ are said to be unitarily equivalent, written as $\phi\sim_u\psi$, if there is a unitary $u\in B$ such that $\phi(a)=u\psi(a)u^*$ for all $a\in A$. If $\varepsilon>0$ and $F\subseteq A$ is a finite set, $\phi$ is said to be approximately unitarily equivalent to $\psi$ on $F$ to within $\varepsilon$, written as $\phi\sim_{F, \,\varepsilon}\psi$, if there is a unitary $u\in B$ such that $\|\phi(a)-u\psi(a)u^*\|<\varepsilon$ for all $a\in F$.
\end{definition}

\begin{definition}[\cite{JS}, \cite{GL}]\label{JSD}
A generalized dimension drop interval algebra, denoted by $I[m_0,m,m_1]$,
is the unital ${\mathrm C}^*$-algebra
$$
I[m_0,m,m_1]=\{f\in \mathrm{M}_m(\mathrm{C}([0,1])):f(0)=a_0\otimes 1_{m/m_0}, f(1)=a_1\otimes 1_{m/m_1}\},
$$
where $m_0,\,m_1$ divide $m$, $a_0$ and $a_1$ (for a given $f$) belong to $\mathrm{M}_{m_0}$ and $\mathrm{M}_{m_1}$, respectively, and
$1_{m/m_0}$ and $1_{m/m_1}$ are the identity elements of $\mathrm{M}_{m/m_0}$ and $\mathrm{M}_{m/m_1}$, respectively.
\end{definition}

We note that the algebras $\widetilde{I}_p$ are generalized dimension drop interval algebras and $\mathrm{M}_{r}(I[m_0,m,m_1])\cong I[rm_0,rm,rm_1]$.

\begin{definition}\label{irrep}
For any $f\in I[m_0,m,m_1]$, $t\in (0,1)$, let us define the point evaluation map $\pi_t: I[m_0,m,m_1]
\rightarrow \mathrm{M}_{m}$ by $\pi_t(f)=f(t)$; for $t=0,1$, write $f(t)=a_t\otimes id_{m/m_t}$, and define $\pi_t: I[m_0,m,m_1])
\rightarrow \mathrm{M}_{m_t}$ by $\pi_t(f)=a_t$. These are all the irreducible representations of $I[m_0,m,m_1]$.
\end{definition}

%\begin{definition}
%Let $A$ be a generalized dimension drop interval algebra, $\phi:M_r\oplus M_s\rightarrow M_k(A)$ be a homomorphism, $e_{ij}$ be the matrix units of $M_r$ and $f_{ij}$ be the matrix units of $M_s$, then denote $\phi$ by
%$$
%\phi(a,b)=a\otimes \phi(e_{11})+b\otimes \phi(f_{11})
%$$
%where $a\in M_r$, $b\in M_s$.
%\end{definition}

\begin{definition}\label{deflarge}
 Let $A$ and $B$ be generalized dimension drop interval algebras, and
 $\phi:\mathrm{M}_r(A)\rightarrow \mathrm{M}_k(B)$ a homomorphism. The induced K-theory map $[\phi]$ will be called $l$-large, if there exists a projection $0\neq p\in \mathrm{M}_k(B)$ such that
 $$
 [\phi]([1_A])\geq l\cdot[p] \;{\rm  in}\;\mathrm{K}_0(B),
 $$
 and $[\phi]$ is called strictly $l$-large if $[\phi]$ is $(l+1)$-large.
\end{definition}

\begin{remark} The K-theoretic information in a generalized dimension drop interval algebra is summarized as follows. For $I[m_0,m,m_1]$, one has the short exact sequence
$$
0\rightarrow \mathrm{M}_m(\mathrm{C}_0(0,1)) \xrightarrow{\iota} I[m_0,m,m_1] \xrightarrow{\pi_0\oplus \pi_1} \mathrm{M}_{m/m_0}\oplus \mathrm{M}_{m/m_1}\rightarrow0,
$$
where $\iota$ is the embedding map.

Then we have the six-term exact sequence
$$
0\to \mathrm{K}_0(I[m_0,m,m_1])\xrightarrow{(\pi_0\oplus\pi_1)_*} \mathbb{Z}\oplus \mathbb{Z}\xrightarrow{(\frac{m}{m_0}, -\frac{m}{m_1})} \mathbb{Z}\xrightarrow{\iota_*} \mathrm{K}_1(I[m_0,m,m_1])\to 0.
$$
Hence,
$$
\mathrm{K}_0(I[m_0,m,m_1])=\mathbb{Z},\quad
\mathrm{K}_1(I[m_0,m,m_1])=\mathbb{Z}_p,
$$
where $p=(m/m_0,m/m_1)$.
\end{remark}

\section{weak variation}
In this section, we gather some results needed later for the study of real rank zero limits of generalized dimension drop interval algebras.

\begin{definition}[\cite{DG}]\label{3.1ell}
 Let $F$ be a finite subset of $M_r(I[m_0,m,m_1])$. The weak variation of $F$ is defined by
 $$
 \omega(F)=\sup\limits_{s,t\in[0,1]}\inf\limits_{u\in U(rm)}\max\limits_{f\in F}\|uf(s)u^*-f(t)\|.
 $$
 \end{definition}

   %if $\kappa:M_r(I[m_0,m,m_1])\rightarrow M_{rm}(C[0,1])$ is the canonical embedding, then $\omega(F)=\omega(\kappa(F))$.
  If $\omega(F)<\varepsilon$, we will say that $F$ is weakly approximately constant to within $\epsilon$.

\begin{lem}\label{homstyle}
  Let $\phi:M_r(I[m_0,m,m_1])\rightarrow M_n$ be a homomorphism.
  There exist integers $s_0,s_1$ with $0\leq s_0\leq m/m_0-1$, $0\leq s_1\leq m/m_1-1$ and a unitary $u\in U(M_n)$ such that $\phi$ is of the standard form
  $$
  \phi(f)=u\cdot{\rm diag}\big(\underbrace{\pi_0(f),\cdots,\pi_0(f)}_{s_0\,\,times},
  \underbrace{\pi_1(f),\cdots,\pi_1(f)}_{s_1\,\,times},\widehat{\phi}(f)\big)\cdot u^*
  $$
  where $\widehat{\phi}$ is the restriction to the subalgebra $M_r(I[m_0$,$m,m_1])$ of some homomorphism $\widehat{\phi}:M_r(M_m(C[0,1]))\rightarrow M_n$. The integers $s_0,s_1$ are uniquely determined by the $\mathrm{KK}$-class of $\phi$.
\end{lem}
 \begin{proof}
 This follows from the representation theory of $I[m_0,m,m_1]$ (see Definition \ref{irrep}) and its KK-theory; see also Lemma 3.5 of \cite{JS}.
 \end{proof}
  Let $\phi:M_r(I[m_0,m,m_1])\rightarrow M_k(I[n_0,n,n_1])$ be a unital homomorphism, for any $t\in (0,1)$, denote $\phi_t$ by $\phi_t(f)=\phi(f)(t)$. Then for any $t_1,t_2\in [0,1]$, $\phi_{t_1}$ is homotopic to $\phi_{t_2}$, and so $\mathrm{KK}(\phi_{t_1})=\mathrm{KK}(\phi_{t_2})$. By Lemma \ref{homstyle}, for each $t\in [0,1]$, there exists $u_t\in M_{kn}$ such that
   $$
  \phi_t(f)=u_t\cdot
  {\rm diag}\big(\underbrace{\pi_0(f),\cdots,\pi_0(f)}_{s_0},
  \underbrace{\pi_1(f),\cdots,\pi_1(f)}_{s_1},
  \widehat{\phi}_t\big)
  \cdot
  u_t^*
  $$
  where $s_0$ and $s_1$ are unique such that $0\leq s_0\leq m/m_0$
  and $0\leq s_1\leq m/m_1$, and
   $\widehat{\phi}_t:M_r(M_m(C[0,1]))\rightarrow M_n$ is of the form
  $$
  \widehat{\phi}_t(f)={\rm diag}\{f(\lambda_1(t)),\cdots,f(\lambda_s(t)\},
  $$
  and $\lambda_1(t),\cdots,\lambda_s(t)\in [0,1]$. Moreover, the $\lambda_i(t)$ if labelled in increasing order are continuous with respect to $t$; see also Corollary 3.6 of \cite{JS}.
  (the unitaries $u_t$ may not be possible to choose continuously.)

  We use the notation $\#(\cdot)$ to denote the cardinal number of the set (counting multiplicity), and we use $\{\delta^{\thicksim k}\}$ to denote
  $\{\underbrace{\delta,\cdots,\delta}_{k}\}$.

\begin{definition}\label{defspv}
  Let $A$ and $B$ be generalized dimension drop interval algebras, and $\phi:M_r(A)\rightarrow M_k(B)$ be a homomorphism. Let us define the spectral homomorphism $\phi_t$ by
  $$
  sp\phi_t=\{\delta_0^{\,\sim \,s_0},\delta_1^{\,\sim\,
  s_1},\lambda_1(t),\cdots,\lambda_s(t)\}.
  $$
  Then for each $t\in [0,1]$, $\#(sp\phi_t\cap[0,1])=s$. Let us define the spectral variation of $\phi$ by
  $$
  spv(\phi)=\max_{t_1,t_2\in [0,1]}dist
  (sp(\phi_{t_1})\cap[0,1],sp(\phi_{t_2})\cap[0,1])
  $$
  where $sp(\phi_t)\cap(0,1)$ is regarded as an element of $P^s([0,1])$, the symmetric Cartesian product of $s$ copies of $[0,1]$. The distance is taken in $P^s([0,1])$ (see \cite{EG}).
\end{definition}

  We need the following lemma, which is a minor generalization of Lemma 3.6 in \cite{DG}.
\begin{lem}\label{wvinq}
  Suppose that $A$ and $B$ are generalized dimension drop interval algebras and let
  $F\subset M_r(A)$ be a finite subset. Then for any homomorphism $\phi:M_r(A)\rightarrow M_k(B)$, one has
  $$
  \omega(\phi(F))\leq\sup_{d(\lambda,\lambda')\leq spv\phi}\max_{f\in F}\|f(\lambda)-f(\lambda')\|,
  $$
  where $d(\lambda,\lambda')$ is the distance between the points $\lambda,\lambda'\in(0,1)$.
\end{lem}

 The following lemma is clear by the standard techniques of spectral theory; see \cite{BBEK}.
\begin{lem}\label{rrpro}
  Let $A=\underrightarrow{lim}(A_n,\nu_{n,m})$ be an inductive limit of C*-algebras $A_n$ with morphisms
  $\nu_{n,m}:A_n\rightarrow A_m$. Then A is of real rank zero if and only if for any finite subset $F\subset (A_n)_{sa}$ and $\varepsilon >0$, there exists $m\geq n$ such
  that for any $r\geq m$, $$\nu_{n,r}(F)\subset_{\epsilon}\{ f\in (A_r)_{sa}\mid\, \text{f has finite spectrum}\}.$$
\end{lem}

   On combining the above lemma with Theorem 2.5 of \cite{Su}, a slight generalization will yield the following conclusion.
\begin{thrm}\label{spsmall}
  Let $A=\underrightarrow{\lim}(A_n,\nu_{n,m})$ be
  a real rank zero $C^*$-algebra inductive limit of direct sums of generalized dimension drop
  interval algebra. Then for any $n$ and any $\delta>0$, there is $m\geq n$ such that $spv(\nu_{n,r})<\delta$ for all $r\geq m$.
\end{thrm}

Combining Lemma \ref{wvinq}, Lemma \ref{rrpro}, and Theorem \ref{spsmall}, we deduce the following corollary.
\begin{cor}\label{wvsmall}
   Let $A=\underrightarrow{\lim}(A_n,\nu_{n,m})$ be
  a real rank zero $C^*$-algebra inductive limit of direct sums of generalized dimension drop
  interval algebras, and let $F\subset A_n$ be a finite set.
  Then $\lim\limits_{r\rightarrow\infty}\omega(\nu_{n,r}(F))=0$.
\end{cor}

 Note that a generalized dimension drop interval algebra has in general two different algebras at the endpoints.
 By Theorem 3.1 and Remark 3.1 in \cite{Liu}, we have the following theorem.
\begin{thrm}\label{decom}
Let $A$ and $B$ be generalized dimension drop interval algebras, $G\subset M_r(A)$ be a finite set, $\epsilon>0$, and $L$ be a positive integer. Then there exists a finite set $G'\subset M_r(A)$ such that if a homomorphism $\nu: M_r(A)\rightarrow M_k(B)$ satisfies
$$
\nu(G')\subset_{1/6}\{ f\in M_k(B)\mid \,\text{f has finite spectrum}\},
$$
then there exist a projection $p\in M_k(B)$ and a unital homomorphism $\lambda: M_r(A)\rightarrow (1-p)M_k(B)(1-p)$ with finite dimensional image such that

  $(1)\,\,[\lambda(1_{M_r(A)})]\geq L\cdot [p]$ in $K_0(B)$,

  $(2)\,\,\|\nu(a)p-p\nu(a)\|<\epsilon\;\text{for all}\; a\in G$, and

  $(3)\,\,\|\nu(a)-p\nu(a)p-\lambda(a)\|<\epsilon\;\text{for all}\; a\in G$.
\end{thrm}
Note that if all the elements of $G$ have norm at most one, then the map
 $\theta(a)=p\nu(a)p$ is $\epsilon$-multiplicative on $G$.
%\begin{remark}
%  From \ref{defspv}, for a homomorphism $\nu: M_r(A)\rightarrow M_k(B)$, we have
%  we have
%  $$
%  \Delta_{\nu}=
%  \left(
%  \begin{array}{c}
%    sp(\pi_0\circ\nu)\cap (0,1) \\
%    sp(\pi_1\circ\nu)\cap (0,1)
%  \end{array}
%  \right).
%  $$

%  In the proof of Theorem 3.1(1) in \cite{Liu}, the construction of $p$ only uses the %spectral in $(0,1)$, so we can change the condition (1) in \ref{decom} to a stronger %version:
%$(1)'\,\,[\lambda(1_{M_r(A)})]\geq rm\cdot\Delta_{\nu} \geq L\cdot [p].$
%%%%\end{remark}

\section{existence results}

In this section, we will prove the local existence theorem for our classification. 

First, we present a concrete picture of KK-groups for generalized dimension drop interval algebras.

\begin{definition}\label{cmset}
Let $A=I[m_0,m,m_1]$ and $B=I[n_0,n,n_1]$ be generalized dimension drop interval algebras. Denote by $C(A,B)$ the set of all the commutative diagrams
  $$
\xymatrixcolsep{2pc}
\xymatrix{
{\,\,0\,\,} \ar[r]^-{}
& {\,\,\mathrm{K}_0(A)\,\,} \ar[d]_-{\lambda_{0*}} \ar[r]^-{(\pi_0\oplus\pi_1)_*}
& {\,\,\mathbb{Z}\oplus \mathbb{Z}\,\,} \ar[d]_-{\lambda_{0}} \ar[r]^-{(\frac{m}{m_0}, -\frac{m}{m_1})}
& {\,\, \mathbb{Z}\,\,} \ar[d]_-{\lambda_{1}} \ar[r]^-{\iota_*}
& {\,\,\mathrm{K}_1(A)\,\,} \ar[d]_-{\lambda_{1*}} \ar[r]^-{}
& {\,\,0\,\,}\\
{\,\,0\,\,} \ar[r]^-{}
& {\,\,\mathrm{K}_0(B)\,\,} \ar[r]_-{(\pi_0'\oplus\pi_1')_*}
& {\,\,\mathbb{Z}\oplus \mathbb{Z} \,\,} \ar[r]_-{(\frac{n}{n_0}, -\frac{n}{n_1})}
& {\,\, \mathbb{Z} \,\,} \ar[r]_-{\iota_*'}
& {\,\,\mathrm{K}_1(B)\,\,} \ar[r]^-{}
& {\,\,0\,\,},}
$$
where all maps are group homomorphisms.

Denote by $M(A,B)$ the subset of $C(A,B)$ of all the commutative diagrams
  $$
\xymatrixcolsep{2pc}
\xymatrix{
{\,\,0\,\,} \ar[r]^-{}
& {\,\,\mathrm{K}_0(A)\,\,} \ar[d]_{0} \ar[r]^-{(\pi_0\oplus\pi_1)_*}
& {\,\,\mathbb{Z}\oplus \mathbb{Z}\,\,} \ar[d]_{\mu_{0}} \ar[r]^-{(\frac{m}{m_0}, -\frac{m}{m_1})}
& {\,\, \mathbb{Z}\,\,} \ar@{.>}[dl]_{\mu} \ar[d]_{\mu_{1}} \ar[r]^-{\iota_*}
& {\,\,\mathrm{K}_1(A)\,\,} \ar[d]_{0} \ar[r]^-{}
& {\,\,0\,\,}\\
{\,\,0\,\,} \ar[r]^-{}
& {\,\,\mathrm{K}_0(B)\,\,} \ar[r]_{(\pi_0'\oplus\pi_1')_*}
& {\,\,\mathbb{Z}\oplus \mathbb{Z} \,\,} \ar[r]_{(\frac{n}{n_0}, -\frac{n}{n_1})}
& {\,\, \mathbb{Z} \,\,} \ar[r]_{\iota_*'}
& {\,\,\mathrm{K}_1(B)\,\,} \ar[r]^-{}
& {\,\,0\,\,}}
$$
such that there exists a group homomorphism $\mu$ satisfying $\mu_0=\mu \circ(\frac{m}{m_0}, -\frac{m}{m_1})$ and $\mu_1=(\frac{n}{n_0}, -\frac{n}{n_1})\circ \mu$ (i.e., two triangles commute). \end{definition}

The set $C(A, B)$ forms an abelian group under the addition of group homomorphisms, namely, we add two diagrams by adding up all corresponding group homomorphisms. $M(A, B)$ is a normal subgroup of $C(A, B)$, and the quotient group is actually isomorphic to $\mathrm{KK}(A, B)$; moreover, composition of diagrams corresponds to the Kasparov product of KK-elements.

This picture holds for general Elliott-Thomsen algebras.

\begin{thrm}[\cite{AE}, Theorem 2.9]\label{CM}
Let there be given two C*-algebras $A$ and $B$ in $\mathcal{C}$, where $\mathcal{C}$ denotes the class of Elliott-Thomsen algebras. Then we have a natural isomorphism of groups:
$$
\mathrm{KK}(A,B)\cong C(A,B)/M(A,B).
$$
\end{thrm}

%The following result is the local existence theorem proved in \cite{AE} and with the fact $\underline{\mathrm{K}}$$(I[m_0,m,m_1])$ is a finite generated $\Lambda$-module (Proposition \ref{dg finite}). But we still need to prove an existence theorem in a stronger sense (i.e., the cone $\underline{\mathrm{K}}^+(I[m_0,m,m_1])$ is also finite generated as $\Lambda$-module), which can be used in the intertwining argument.

%\begin{thrm}[\cite{AE}, Theorem 4.1]\label{EXJS}
%Let $A=I[m_0,m,m_1],\,B=I[n_0,n,n_1]$.
%If $\alpha$ is a $\mathrm{KK}$-element in $\mathrm{KK}(A,B)$ satisfying
%$$
%\alpha(\underline{\mathrm{K}}^+(A))\subseteq \underline{\mathrm{K}}^+(B),
%$$
%then $\alpha$ can be lifted to a homomorphism from $A$ to $B\otimes \mathcal{K}$.
%\end{thrm}

%\begin{remark}\label{exrmak} We point out that in the proof of Theorem \ref{EXJS}, actually it is shown that as long as $\alpha([\iota_m])\in \underline{\mathrm{K}}^+(B)$, then $\alpha$ can be lifted to a homomorphism into $B\otimes \mathcal{K}$, where $\iota_m: \tilde{I}_m\rightarrow I[m_0, m, m_1]$ is the natural embedding.\end{remark}
Next we focus on the problem of lifting KK-elements into homomorphisms of algebras.
\begin{prop}\label{scale} Let $A$ and $B$ be two generalized dimension drop
interval algebras. If $\alpha\in \mathrm{KK}(A,B)$ can be lifted to a $*$-homomorphism $\phi$ from $A$ to $M_r(B)$ for some integer $r$, and $\alpha_*([1_{A}])\leq[1_{B}]$, then $\alpha$ can be lifted to a $*$-homomorphism from $A$ to $B$. In particular, if $\alpha_*([1_{A}])=[1_{B}]$, we can lift it to a unital $*$-homomorphism.
\end{prop}
\begin{proof}
As $\phi$ is a $*$-homomorphism from $A$ to $B\otimes
\mathcal {K}(H)$, such that $[\phi]=\alpha$, since
$\alpha_*([1_{A}])\leq[1_{B}]$, there is a projection $q\in
B_n$, such that $[\phi(1_{A})]_0=[q]_0$. Because dimension drop interval algebras have cancellation of projections, there is a partial isometry $v$ in $B\otimes \mathcal{K}(H)$, such that $\phi(1_{A})=v^*v$ and $q=vv^*$. Define $\widetilde{\phi}=v\phi v^*$.
Then for any $a\in A$, one has $\widetilde{\phi}(a)=v\phi(a)v^*=vv^*v\phi(a)v^*vv^*=qv\phi(a)v^*q\in B$. So $\widetilde{\phi}$ is a homomorphism
from $A$ to $B$, and $\widetilde{\phi}$ is a lifting of $\alpha$.
\end{proof}
The following lemma is a special case of Theorem 3.8 in \cite{AE}.
\begin{lem} \label{CRCITE}
Let $A=I[m_0,m,m_1]$ and $B=I[n_0,n,n_1]$.
A commutative diagram $\lambda\in C(A,B)$
$$
\xymatrixcolsep{2pc}
\xymatrix{
{\,\,0\,\,} \ar[r]^-{}
& {\,\,\mathrm{K}_0(A)\,\,} \ar[d]_-{\lambda_{0*}} \ar[r]^-{(\pi_0\oplus\pi_1)_*}
& {\,\,\mathbb{Z}\oplus \mathbb{Z}\,\,} \ar[d]_-{\lambda_{0}} \ar[r]^-{(\frac{m}{m_0}, -\frac{m}{m_1})}
& {\,\, \mathbb{Z}\,\,} \ar[d]_-{\lambda_{1}} \ar[r]^-{\iota_*}
& {\,\,\mathrm{K}_1(A)\,\,} \ar[d]_-{\lambda_{1*}} \ar[r]^-{}
& {\,\,0\,\,}\\
{\,\,0\,\,} \ar[r]^-{}
& {\,\,\mathrm{K}_0(B)\,\,} \ar[r]_-{(\pi_0'\oplus\pi_1')_*}
& {\,\,\mathbb{Z}\oplus \mathbb{Z} \,\,} \ar[r]_-{(\frac{n}{n_0}, -\frac{n}{n_1})}
& {\,\, \mathbb{Z} \,\,} \ar[r]_-{\iota_*'}
& {\,\,\mathrm{K}_1(B)\,\,} \ar[r]^-{}
& {\,\,0\,\,}}
$$
can be lifted to a homomorphism (from $A$ to $B\otimes\mathcal{K}$) if, and only if, $\lambda_0$ has no negative entries.
\end{lem}
The following result is our local existence theorem for the classification.

\begin{thrm}\label{exist} Let $A$ be a generalized dimension drop interval algebra, then
there exists a finite subset $F$ of $\underline{\mathrm{K}}^+(A)$ such that, for any generalized dimension drop interval
algebra $B$, if $\alpha\in \mathrm{KK}(A,B)$ satisfies the
conditions

$(1)\, \alpha_{*}(F)\subseteq \underline{\mathrm{K}}^+(B),$

$(2)\, \alpha_{*}[1_{A}]\leq[1_{B}],$\\
then there is a $*$-homomorphism $\phi:A\rightarrow B$ which lifts $\alpha$.
\end{thrm}
\begin{proof}
%Recall that $A=I[m_0,m,m_1]$. Denote by $\iota_m$ the natural embedding from $\widetilde{I}_m$ to $A$, and set $F=\{\mathrm{KK}(\iota_m)\}\subseteq \underline{\mathrm{K}}^+(A)$. Then Theorem \ref{EXJS} and Remark \ref{exrmak} show that $\alpha$ can be lifted to a $*$-homomorphism from $A$ to $M_r(B)$ for some $r$.
Recall that $A=I[m_0,m,m_1]$. Denote by $\iota_m$ the natural embedding from $\widetilde{I}_m$ to $A$, and set $F=\{\mathrm{KK}(\iota_m)\}\subseteq \underline{\mathrm{K}}^+(A)$. Note that $\iota_m$ induces the following diagram $\kappa\in C(\widetilde{I}_m, A)$
$$\xymatrixcolsep{2pc}
\xymatrix{
{\,\,0\,\,} \ar[r]^-{}
& {\,\,\mathrm{K}_0(\widetilde{I}_m)\,\,} \ar[d]_-{1} \ar[r]^-{}
& {\,\,\mathbb{Z}\oplus \mathbb{Z}\,\,} \ar[d]_-{\fontsize{8pt}{5pt}\selectfont\left(\!\!
\begin{array}{cc}
m_0&\\[2 mm]
&m_1
\end{array}\!\!\right)} \ar[r]^-{(m,-m)}
& {\,\, \mathbb{Z}\,\,} \ar[d]_-{1} \ar[r]^-{}
& {\,\,\mathrm{K}_1(\widetilde{I}_m)\,\,} \ar[d]_-{1} \ar[r]^-{}
& {\,\,0\,\,}\\
{\,\,0\,\,} \ar[r]^-{}
& {\,\,\mathrm{K}_0(A)\,\,} \ar[r]_-{(\pi_0\oplus\pi_1)_*}
& {\,\,\mathbb{Z}\oplus \mathbb{Z} \,\,} \ar[r]_-{(\frac{m}{m_0}, -\frac{m}{m_1})}
& {\,\, \mathbb{Z} \,\,} \ar[r]_-{\iota_*}
& {\,\,\mathrm{K}_1(A)\,\,} \ar[r]^-{}
& {\,\,0\,\,}.}
$$
By Theorem \ref{CM}, there exists a commutative diagram $\lambda\in C(A,B)$
$$\xymatrixcolsep{2pc}
\xymatrix{
{\,\,0\,\,} \ar[r]^-{}
& {\,\,\mathrm{K}_0(A)\,\,} \ar[d]_-{} \ar[r]^-{(\pi_0\oplus\pi_1)_*}
& {\,\,\mathbb{Z}\oplus \mathbb{Z}\,\,} \ar[d]_-{\fontsize{8pt}{5pt}\selectfont\left(\!\!
\begin{array}{cc}
a&b\\[2 mm]
c&d
\end{array}\!\!\right)} \ar[r]^-{(\frac{m}{m_0}, -\frac{m}{m_1})}
& {\,\, \mathbb{Z}\,\,} \ar[d]_-{s} \ar[r]^-{\iota_*}
& {\,\,\mathrm{K}_1(A)\,\,} \ar[d]_-{} \ar[r]^-{}
& {\,\,0\,\,}\\
{\,\,0\,\,} \ar[r]^-{}
& {\,\,\mathrm{K}_0(B)\,\,} \ar[r]_-{(\pi_0'\oplus\pi_1')_*}
& {\,\,\mathbb{Z}\oplus \mathbb{Z} \,\,} \ar[r]_-{(\frac{n}{n_0}, -\frac{n}{n_1})}
& {\,\, \mathbb{Z} \,\,} \ar[r]_-{\iota_*'}
& {\,\,\mathrm{K}_1(B)\,\,} \ar[r]^-{}
& {\,\,0\,\,},}
$$
such that $\mathrm{KK}(\lambda)=\alpha$.

Since
$$
\alpha(F)\subset \underline{\mathrm{K}}^+(B),
$$
we have $$\alpha([\iota])\in \mathrm{KK}^+(\widetilde{I}_m,B).$$
Moreover, we also have $\alpha([\iota])=\mathrm{KK}(\kappa\times\lambda)$,
where $\kappa\times\lambda\in C(\widetilde{I}_m,B)$ is the following diagram
$$\xymatrixcolsep{2pc}
\xymatrix{
{\,\,0\,\,} \ar[r]^-{}
& {\,\,\mathrm{K}_0(\widetilde{I}_m)\,\,} \ar[d]_-{} \ar[r]^-{}
& {\,\,\mathbb{Z}\oplus \mathbb{Z}\,\,} \ar[d]_-{\fontsize{8pt}{5pt}\selectfont\left(\!\!\begin{array}{cc}
am_0&bm_1\\[1.5 mm]
cm_0&dm_1
\end{array}\!\!\right)} \ar[r]^-{(m,-m)}
& {\,\, \mathbb{Z}\,\,} \ar[d]_-{s} \ar[r]^-{}
& {\,\,\mathrm{K}_1(\widetilde{I}_m)\,\,} \ar[d]_-{} \ar[r]^-{}
& {\,\,0\,\,}\\
{\,\,0\,\,} \ar[r]^-{}
& {\,\,\mathrm{K}_0(B)\,\,} \ar[r]_-{(\pi_0'\oplus\pi_1')_*}
& {\,\,\mathbb{Z}\oplus \mathbb{Z} \,\,} \ar[r]_-{(\frac{n}{n_0}, -\frac{n}{n_1})}
& {\,\, \mathbb{Z} \,\,} \ar[r]_-{\iota_*'}
& {\,\,\mathrm{K}_1(B)\,\,} \ar[r]^-{}
& {\,\,0\,\,}.}
$$
Note that any homomorphism between generalized dimension drop algebras is homotopic to a homomorphism which induces a commutative diagram, combining this fact with Theorem \ref{CM} and Lemma \ref{CRCITE}, there exists a $\lambda_\mu\in M(\widetilde{I}_m,B)$ such that $\mu_0+\fontsize{8pt}{5pt}\selectfont\left(\!\!\begin{array}{cc}
am_0&bm_1\\[1.5 mm]
cm_0&dm_1
\end{array}\!\!\right)$ has no negative entries.

That is, there exist $u_1,\,u_2\in \mathbb{Z}$ such that
$$
am_0+u_1m\geq0,\quad
bm_1-u_1m\geq0,
$$
$$
cm_0+u_2m\geq0,\quad
dm_1-u_2m\geq0.
$$
Then we have
$$
a+u_1m/m_0\geq0,\quad
b-u_1m/m_1\geq0,
$$
$$
c+u_2m/m_0\geq0,\quad
d-u_2m/m_1\geq0.
$$
%Note that
%$$
%{\left(\begin{array}{ccc}
%a+u_1m/m_0&b-u_1m/m_1\\
%c+u_1m/m_0&d-u_2m/m_1
%\end{array}\right)}\neq0.
%$$
So from Lemma \ref{CRCITE}, we know that $\alpha$ can be lifted to a homomorphism from $A$ to $M_r(B)$ for some integer $r$, then by Proposition \ref{scale}, $\alpha$ can be lifted to a $*$-homomorphism from $A$ to $B$.

\end{proof}

\begin{cor}\label{exists} A similar statement is true for direct sums
of generalized dimension drop interval algebras.
\end{cor}

\begin{prop}[\cite{DG}, Proposition 4.13]\label{dg finite}
Let $A$ be a $C^*$-algebra in the class $\mathcal{N}$ of \cite{RS}. If the group $\mathrm{K}_*(A)$ is finitely
generated, then $\underline{\mathrm{K}}(A)$ is finitely generated as $\Lambda$-module.
(In other words, there are finitely many elements $x_1,\cdots,x_r\in\underline{\mathrm{K}}(A)$ such that for any
$x\in\underline{\mathrm{K}}(A)$ there exist $\lambda_i\in \Lambda$ and $k_i\in \mathbb{Z}$ such that $x=\sum_{i=1}^r k_i \lambda_i(x_i)$.)
\end{prop}
The following theorem is a corollary of the universal coefficient theorem list in \cite{DG}.
\begin{thrm}\label{UCT}
Let $A$ and $B$ be ${\mathrm C}^*$-algebras. Suppose that $A\in \mathcal{N}$ (where $\mathcal{N}$ is
the ``bootstrap'' category defined in \cite{RS}), $\mathrm{K}_*(A)$ is finitely generated, and $B$ is $\sigma$-unital.
Then the natural map
$$\Gamma:\,\mathrm{KK}(A,B)\to {\bf Hom}_\Lambda(\underline{\mathrm{K}}(A),\underline{\mathrm{K}}(B))$$ is a group isomorphism.
\end{thrm}

%\begin{thrm}\label{shape}
%Let $A=\underrightarrow{\lim}(A_n,\phi_{n,m})$ and $B=\underrightarrow{\lim}(B_n,\psi_{n,m})$ be
%two $C^*$-algebra inductive limits of direct sums of generalized dimension drop
%interval algebras. Suppose there is an isomorphism of ordered scaled groups
%$$\alpha: (\underline{\mathrm{K}}(A), \underline{\mathrm{K}}^+(A), \Sigma(A))
%\rightarrow
%(\underline{\mathrm{K}}(B), \underline{\mathrm{K}}^+(B), \Sigma(B))$$ which preserves the action of the Bockstein operations. Then $A$ and $B$ are KK-shape equivalent.
%\end{thrm}

\section{uniqueness results and the classification}
In this section, we establish some uniqueness results, in particular Theorem \ref{unithm0}, which is crucial for real rank zero limits built on generalized dimension drop interval algebras, and then obtain the asserted isomorphism theorem.

\begin{definition}
A C*-algebra $A$ is said to have property (H) if for any $\epsilon>0$ and any finite subset $F\subset A$, there exist a natural number $r\in \mathbb{N}$, a $*$-homomorphism $\tau: A\rightarrow M_{r-1}(A)$, and a $*$-homomorphism $\mu: A\rightarrow M_r(A)$ with finite dimensional image, such that $\|a\oplus \tau(a)-\mu(a)\|<\epsilon$ for all $a\in F$.

\end{definition}

\begin{thrm}\label{prouni}
Let $A$ be a generalized dimension drop interval
algebra, and let B be a unital C*-algebra. Given two homomorphisms
$\phi, \psi: M_r(A)\rightarrow B$, suppose that $[\phi]=[\psi]$ in $\mathrm{KK}(A,
B)$. Then for any finite subset $F\subset M_r(A)$ and any
$\varepsilon>0$, there exists a natural number $k$, a $*$-homomorphism
$\eta: M_r(A)\rightarrow M_{k}(B)$ with finite dimensional range, and a unitary $v\in M_{k+1}(B)$
such that
$$\|v(\phi(f)\oplus\eta(f))v^*-\psi(f)\oplus\eta(f)\|<\varepsilon$$ for all $f\in F$.
\end{thrm}
\begin{proof} Since the KK-type of homomorphisms is the same as the homotopy type of them, the statement follows from Lemma 1.4 of \cite{D}, provided one shows that a generalized dimension drop interval algebra $A$ also has the property (H) (and so also its matricial stabilization does). This is seen as follows.
In the notation of \cite{GL}, using canonical homomorphisms $id$ and $\overline{id}$ (corresponding to identity eigenvalue map and the map switching endpoints), one can construct a real rank zero inductive limit system $D=\lim\limits_{\longrightarrow}(M_{k_i}(A), \phi_i)$, where $\phi_i$ is given in terms of $id, \overline{id}$ and certain homomorphism with finite dimensional image,
such that the connecting maps vanish on $\mathrm{K}_1$. On the other hand, a generalized dimension drop interval algebra is semiprojective (see \cite{EiLP}), and so the argument of stable relations works, and then mimic the proof of Theorem 1.4 of \cite{DL3}, one gets that the inclusion of $A$ into the limit $D$
can be approximated arbitrarily well by a homomorphism $\mu$ with finite
dimensional range. Then applying a standard perturbation argument, for any $\epsilon>0$ and any $F\subset A$, one can pull $\mu$ back to certain finite steps, within $\epsilon$ on $F$, say $\mu_i$, and the inclusion map of $A$ into $D$ obviously has the form $a\oplus \tau$ for some homomorphism $\tau$.
 \end{proof}

\begin{lem}\label{firsthop}
Let $A$ be a generalized dimension drop interval
algebra, and consider the matrix algebra (tensor product) $M_r(A)$. For any finite
subset $F\subset M_r(A)$, for any $\varepsilon>0$, there exist a finite subset $E\subset M_r(A)$ and $\delta>0$,
such that whenever B is a unital C*-algebra and whenever $\Phi\in Map(M_r(A), B[0,1])_1$ is $\delta$-multiplicative
on E, there is a natural number $L$, a unital homomorphism $\lambda: M_r(A)\rightarrow M_L(B)$ with finite dimensional range
and a unitary $u\in M_{L+1}(B)$ such that $$\|u\Phi_0(f)\oplus\lambda(f)u^*-\Phi_1(f)\oplus\lambda(f)\|<\varepsilon$$
 for all $f\in F$.
\end{lem}
\begin{proof} Since $A$ has property (H), so does $M_r(A)$. Then
apply Lemma 1.4 of \cite{D} to the sequence $\phi_j=\Phi_{j/m}, 0\leq j\leq m$, where $m$ is large enough
that $\|\phi_{j+1}(f)-\phi_j(f)\|<\epsilon$
for all $f\in F$ and $0\leq j\leq m$.
\end{proof}
 Lemma \ref{simcase1} and Lemma \ref{simcase2} are minor generalization of Lemma 6.3 and Lemma 6.4 of \cite{DG} with the same proof.
\begin{lem}\label{simcase1}
  Let $A$ be a generalized dimension drop interval algebra, $\lambda,\lambda': M_r\oplus M_s\rightarrow M_k(A)$ be two unital homomorphisms inducing the same map on $K_0$. Then there is a unitary $u\in M_k(A)$ such that $u\lambda u^*=\lambda'$.
\end{lem}
\begin{lem}\label{simcase2}
  Let $A$ be a generalized dimension drop interval algebra, $\lambda,\lambda': M_r\oplus M_s\rightarrow M_k(A)$ be two homomorphisms. Suppose that $\lambda$ is unital and $[\lambda(e_{11},0)]\geq [\lambda'(e_{11},0)]$, $[\lambda(0,f_{11})]\geq [\lambda'(0,f_{11})]$. Then there are a homomorphism $\eta:M_r\oplus M_s\rightarrow M_k(A)$ and a unitary $u\in M_k(A)$ such that $u\lambda u^*=\lambda'+\eta$.
\end{lem}

\begin{lem}\label{unires1}
  Let $A$ and $B$ be generalized dimension drop interval algebras, $F\subset M_r(A)$ be a finite set, and let
  $\lambda,\lambda': M_r(A)\rightarrow M_k(B)$ be unital homomorphisms with finite dimensional image such that $[\lambda]=[\lambda']\in \mathrm{KK}(A,B)$. Then there exists a unitary $u\in M_k(B)$ such that
  $$
  \|u^*\lambda(f)u-\lambda'(f)\|<3\omega(F).
  $$
\end{lem}
\begin{proof}
  Since $\lambda$ has finite dimensional image, there are $t_1,\cdots,t_s\in (0,1)$ such that $\lambda$ factors through
  $$
  \pi_{t_1}\oplus\cdots\oplus\pi_{t_s}\oplus\pi_0\oplus\pi_1:M_r(A)\rightarrow (\bigoplus_{i=1}^sM_{rm})\oplus M_{rm_0}\oplus M_{rm_1}.
  $$
  By the definition of weak variation, there are unitaries $u_i\in M_{rm}$ such that
  $$
  \|u_i\pi_{t_i}(f)u_i^*-\pi_0(f)\otimes 1_{m/m_0}\|<\omega(F).
  $$
  In particular, there exists $u_0\in M_{rm}$ such that
  $$
  \|u_0\pi_0(f)\otimes 1_{m/m_0}u_0^*-\pi_1(f)\otimes 1_{m/m_1}\|<\omega(F).
  $$

  It follows that there exist a unitary $v\in M_k(B)$ and  a unital homomorphism $\widetilde{\lambda}:M_r(A)\rightarrow M_k(B)$ with finite dimensional image such that
  $$
  \|\lambda(f)-v\widetilde{\lambda}(f)v^*\|<\omega(F)
  $$
  and $\widetilde{\lambda}$ factors through $\pi_0\oplus\pi_1:M_r(A)\rightarrow M_k(B)$. In addition, since $\pi_{t_i}(f)$ is homotopic to $\pi_0(f)\otimes 1_{m/m_0}$, we have $[\lambda]=[\widetilde{\lambda}]$ in $\mathrm{KK}(A,B)$. Of course we can approximate $\lambda'$ by a homomorphism $\widetilde{\lambda'}$ with similar properties.

   The restrictions of $\widetilde{\lambda}$ and $\widetilde{\lambda'}$ to $M_r(\mathbb{C})\subset M_r(A)$ induce the same map on $K_0$, and so by Lemma \ref{simcase1}, there is a unitary $u\in M_k(B)$ such that $u\widetilde{\lambda}(a) u^*=\widetilde{\lambda'}(a)$ for all $a\in M_r(\mathbb{C})$. We may assume that $\widetilde{\lambda}$ and $\widetilde{\lambda'}$ coincide on $M_r(\mathbb{C})$.

  Since $\widetilde{\lambda}$ and $\widetilde{\lambda'}$ factor through
  $\pi_0\oplus\pi_1$, there are homomorphisms $\phi,\phi':M_{rm_0}\oplus M_{rm_1}\rightarrow M_k(B)$ such that $\widetilde{\lambda}=\phi\circ\pi_e$, $\widetilde{\lambda'}=\phi'\circ\pi_e$.
  Write $\phi,\phi'$  as
 $$
 \phi(a,b)=a\otimes \phi(e_{11},0)+b\otimes \phi(0,f_{11}),
 $$
 $$
 \phi'(a,b)=a\otimes \phi'(e_{11},0)+b\otimes \phi'(0,f_{11}).
 $$
 Then $\phi(e_{11},0),\phi(0,f_{11}),\phi'(e_{11},0),\phi'(0,f_{11})\in M_k(B)$ are projections, and so there exist non-negative integers $k_1,k_2,k_1',k_2'$ such that
$$
[\phi(e_{11},0)]=
k_1\cdot
  \left(
  \begin{array}{c}
     n_0 \\
    n_1
   \end{array}
  \right),\quad
[\phi(0,f_{11})]=
k_2\cdot
  \left(
  \begin{array}{c}
     n_0 \\
    n_1
   \end{array}
  \right),
$$$$
[\phi'(e_{11},0)]=
k_1'\cdot
  \left(
  \begin{array}{c}
     n_0 \\
    n_1
   \end{array}
  \right),\quad
[\phi'(0,f_{11})]=
k_2'\cdot
  \left(
  \begin{array}{c}
     n_0 \\
    n_1
   \end{array}
  \right).
$$

Since $\phi,\phi'$ is unital, we have
$$
m_0k_1+m_1k_2=m_0k_1'+m_1k_2'=k/r,
$$
$$
(k_1-k_1')m_0=(k_2'-k_2)m_1.
$$

If $k_1=k_1'$, then $k_2=k_2'$, and by Lemma \ref{simcase1}, there exists a unitary $v\in M_k(B)$, such that
$$
v^*\phi v=\phi'.
$$
And for the same unitary, we have
$$
\|v^*\widetilde{\lambda}(f)v-\widetilde{\lambda'}(f)\|<\omega(F).
$$

Without loss of generality, we may assume that $k_1> k_1'\geq 0$. Since $[\widetilde{\lambda}]$=$[\widetilde{\lambda'}]$,  by Theorem \ref{CM}, there exists an integer $l>0$ such that
 $$
  \left(
  \begin{array}{c}
     k_1n_0-k_1'n_0 \\
     k_2n_1-k_2'n_1
   \end{array}
  \right)=
  l\cdot
  \left(
  \begin{array}{c}
     m/m_0 \\
    -m/m_1
   \end{array}
  \right).
  $$
Now we have
$$
(k_1-k_1')n_0=l\cdot m/m_0,
\quad
(k_2'-k_2)n_1=l\cdot m/m_1.
$$
Then
$$
(k_1-k_1')m_0n_1=(k_2'-k_2)n_1m_1=lm_0\cdot m/m_0,
$$
$$
(k_2'-k_2)m_1n_0=(k_1-k_1')n_0m_0=lm_1\cdot m/m_1.
$$
We have
$$
(k_1-k_1')\cdot
  \left(
  \begin{array}{c}
     n_0 \\
    n_1
   \end{array}
  \right)
= l\cdot
  \left(
  \begin{array}{c}
    m/m_0 \\
    m/m_0
   \end{array}
  \right),
$$
$$
(k_2'-k_2)\cdot
  \left(
  \begin{array}{c}
     n_0 \\
    n_1
   \end{array}
  \right)
= l\cdot
  \left(
  \begin{array}{c}
    m/m_1 \\
    m/m_1
   \end{array}
  \right).
$$
Then there exist projections $p_0,q_0\in M_k(B)$ such that $p_0,q_0\leq \phi(e_{11})$ and
$$
[p_0]=
(k_1-k_1')
\left(
  \begin{array}{c}
     n_0 \\
    n_1
   \end{array}
  \right),\quad[q_0]=
(k_2'-k_2)\cdot
\left(
  \begin{array}{c}
     n_0 \\
    n_1
   \end{array}
  \right).
$$
By the definition of $\omega(F)$, it follows that
$$
\|w\pi_0(f)\otimes p_0w^*-\pi_1(f)\otimes q_0\|<\omega(F)
$$
for some unitary $w\in M_k(B)$. Then we can construct a homomorphism $\phi'':M_{rm_0}\oplus M_{rm_1}\rightarrow M_k(B)$
such that
$$
\phi''(a,b)=a\otimes (\phi(e_{11},0)-p_0)+b\otimes (\phi(0,f_{11})+q_0)
$$
and
$$
[\phi''(e_{11},0)]=[\phi'(e_{11},0)],\quad[\phi'(0,f_{11})]=[\phi''(0,f_{11})].
$$
Then there exists a unitary $v'\in M_k(B)$ such that
$$
\|{v'}^*\widetilde{\lambda}(f)v'-\widetilde{\lambda'}(f)\|<2\omega(F).
$$

Recall that
$$
\lambda\sim_{F,\omega(F)}\widetilde{\lambda},\quad \lambda'\sim_{F,\omega(F)}\widetilde{\lambda'}.
$$
It follows there exists a unitary $u\in M_k(B)$ such that
  $$
  \|u^*\lambda(f)u-\lambda'(f)\|<3\omega(F).
  $$
\end{proof}

\begin{lem}\label{unires2}
  Let $A=I[m_0,m,m_1]$, let $F\subset M_r(A)$ be a finite set, and let
  $\lambda,\lambda': M_r(A)\rightarrow M_k(B)$ be homomorphisms with finite dimensional image, where $B$ is a generalized dimension drop interval algebra. Suppose that $\lambda$ is unital and $[\lambda]-[\lambda']$ is $m$-large. Then there exist a unitary $u\in M_k(B)$ and a homomorphism $\eta$ with finite dimensional image such that
  $$
  \|u^*\lambda(f)u-\lambda'(f)\oplus\eta(f)\|<4\omega(F).
  $$
\end{lem}
\begin{proof}
 Assume that $\phi,\phi'$ are as in Lemma \ref{unires1}. In this case, we have
$$
m_0k_1+m_1k_2-(m_0k_1'+m_1k_2')\geq m.
$$

Next, we assume that $0\leq k_1,k_1'\leq m/m_0-1$. If $k_1\geq m/m_0$, there exists an integer $k'$ such that
$$
m/m_0|(k_1-k'),\quad 0\leq k_1-k'\leq m/m_0-1.
$$
By the the same technique as in the proof of Lemma \ref{unires1}, we can construct a homomorphism $\phi'':M_{rm_0}\oplus M_{rm_1}\rightarrow M_k(B)$
such that
$$
[\phi''(e_{11},0)]=
(k_1-k')\cdot
  \left(
  \begin{array}{c}
     n_0 \\
    n_1
   \end{array}
  \right)
$$
and
$$
\|v\phi''\circ\pi_e(f)v^*-\phi\circ\pi_e(f)\|<\omega(F)
$$
for some unitary $v\in M_k(B)$. We may use $k_1-k'$ instead of $k_1$, and the same for $k_1'$, and then,
up to an adjustment of $2\omega(F)$, we may assume that $0\leq k_1,k_1'\leq m/m_0-1$.

We have the following two cases:

Case 1: If $k_1\geq k_1'$, then $k_2\geq k_2'$. We can construct a homomorphism $\eta':M_{rm_0}\oplus M_{rm_1}\rightarrow M_k(B)$ whose image is orthogonal to the image of $\phi'$.

Set
$$
\eta'(a,b)=a\otimes (\phi(e_{11})-\phi'(e_{11}))+b\otimes (\phi(f_ {11})-\phi'(f_{11})).
$$
Obviously, there exists a unitary $v\in M_k(B)$ such that
$$
v^*\phi v=\phi'+\eta'.
$$
Let $\eta:M_r(A)\rightarrow M_k(B)$ be a homomorphism extending $\eta'$; then for the same unitary, we have
$$
\|v^*\lambda(f)v-\lambda'(f)\oplus\eta(f)\|<3\omega(F).
$$

Case 2: If $k_1\leq k_1'$, then $m_1k_2\geq m_1k_2'+m$, and we have
$$
k_1+m/m_0\geq k_1'
\quad
{\rm and}
\quad
k_2-m/m_1\geq k_2'.
$$
Then there exist a unitary $v\in M_k(B)$  and a homomorphism $\phi'':M_{rm_0}\oplus M_{rm_1}\rightarrow M_k(B)$
such that
$$
\pi_e\circ\phi(f)\sim_{F,\omega(F)}\pi_e\circ\phi''(f),
$$
and
$$
[\phi''(e_{11},0)]\geq[\phi'(e_{11},0)],\quad
[\phi''(0,f_{11})]\geq[\phi'(0,f_{11})].
$$
By case 1, there exist a unitary $v\in M_k(B)$ and a homomorphism  $\eta:M_r(A)\rightarrow M_k(B)$ with finite dimensional image such that
$$
\|{v'}^*\phi''\circ \pi_e(f)v'-\lambda'(f)-\eta(f)\|<\omega(F).
$$

Combining these two cases, we have
$$
\|u^*\lambda(f) u-\lambda'(f)-\eta(f)\|<4\omega(F)
$$
for some unitary $u\in M_k(B)$.

\end{proof}

Now we are ready to prove the following uniqueness theorem which is crucial for the classification.
\begin{thrm}\label{unithm0}
Let $A$ be a finite direct sum of matrix algebras over generalized dimension drop interval algebras, and let
$B=\underrightarrow{\lim}(B_n,\nu_{n,m})$ be a $C^*$-algebra inductive limit
of finite direct sums of generalized dimension drop interval algebras. Suppose that B is of real rank zero.
Let
$\phi,\psi:A\rightarrow B_n$ be two homomorphisms with $[\phi]=[\psi]$ in $\mathrm{KK}(A,B_n)$. Then for any finite set $F\subset A$ and any
$\varepsilon>0$, there exist $r\geq n$ and a unitary $u\in B_r$ such that
$$
\|u(\nu_{n,r}\circ\phi)(f)u^*-(\nu_{n,r}\circ\psi)(f)\|<3\varepsilon+\omega(F)
$$
for all $f\in F$, where $\omega(F)$ is as in \ref{3.1ell}.
\end{thrm}
\begin{proof}
  Since $A$ has cancellation of projections and $[\phi(1)]=[\psi(1)]\in \mathrm{K}_0(B_n)$, after increasing $n$ we find a unitary $u\in B_n$ such that $u\phi(1)u^*=\psi(1)$.
  Therefore we may assume that $\phi(1)=\psi(1)=p$ by replacing the system $(B_r)$ with $(\nu_{n,r}(p)B_r\nu_{n,r})$ and changing notation. We divide the rest of proof into several parts.

  (a) Let the finite set $F\subset A$ and $\varepsilon>0$ be fixed. We may assume that $F$ contains the units of all the direct summands of $A$. Let $\varepsilon_2>0$. Since $\mathrm{KK}(\phi)=\mathrm{KK}(\psi)$, it follows from Theorem \ref{prouni} that there exists an integer $k$, a homomorphism $\eta:A\rightarrow M_k(B_n)$ with finite dimensional image and a unitary $v\in M_{k+1}(B_n)$ such that
 $$
 \|v(\phi(f)\oplus\eta(f))v^*-\psi(f)\oplus\eta(f)\|<\varepsilon_2
 $$
 for all $f\in F$.

 (b) Since the image of $\eta$ is a semiprojective $C^*$-algebra, it follows that
 there exist a finite set $G_1\subset B_n$ with all elements of norm at most one and $\varepsilon_1>0$ such that if $C$ is a unital $C^*$-algebra and $\theta:\,B_n\to C$ is unital and $\varepsilon_1$-multiplicative on $G_1$, then there exists a unitary $w\in M_{k+1}(C)$ and a homomorphism $\mu:A\rightarrow M_k(C)$ with finite dimensional image such that
 $$
 \|(\theta\otimes 1_{k+1})(v)-w\|<\varepsilon_2(1+M)^{-1},
 $$
 where $M=\sup\{\|f\|\,:\,f\in F\},$ and
 $$
 \|(\theta\otimes 1_k)\eta(f)-\mu(f)\|<\varepsilon_2
 $$
 for all $f\in F$. Let $\theta_s$ denote the map $\theta\otimes 1_s$. By enlarging $G_1$, we may arrange that
 $$
 \|(\theta_{k+1})(v(\phi(f)\oplus\eta(f))v^*)-
 (\theta_{k+1})(v)(\theta_{k+1})(\phi(f)\oplus\eta(f))(\theta_{k+1})(v^*)\|<\varepsilon_2
 $$
 for all $f\in F$. Since $\theta_s$ is contractive, we obtain that
 $$
 \|(\theta_{k+1})(v)\theta\phi(f)\oplus\theta_k\eta(f)(\theta_{k+1})(v^*)-
 \theta\psi(f)\oplus\theta_k\eta\|<2\varepsilon_2.
 $$
 Combining this with the estimate above, we have
 $$
 \|w(\theta\phi(f)\oplus\mu(f))w^*-
 \theta\psi(f)\oplus\mu(f)\|<6\varepsilon_2
 $$
 for all $f\in F$. By the continuous functional calculus, we can find a unitary $w_1\in M_{k+1}(C)$ which commutes with $\theta\phi(1)\oplus\mu(1)$ and $\|w-w_1\|<f(\varepsilon_2)$.
 If $\varepsilon_2$ is chosen small enough to ensure that $2Mf(\varepsilon_2)+6\varepsilon_2<\varepsilon$, then we may replace $w$ by $w_1$ in the above estimate, so that
 \begin{equation}\label{}
   \|w_1(\theta\phi(f)\oplus\mu(f))w_1^*-
 \theta\psi(f)\oplus\mu(f)\|<\varepsilon
 \end{equation}
 for all $f\in F$.

 (c) Since $B_n$ is a direct sum of algebras $B_{n,i}$ with connected spectrum, it suffices to consider the case the spectrum of $A$ is connected. Let $A=M_r(I[m_0,m,m_1])$ and let $e$ denote the unit of $A$. We may assume that $\phi(e)\neq 0$.

 Let $G_1$ and $\varepsilon_1<\varepsilon$ be as in (b) and let $L=k+rm$ with $k$ as defined in (a). Let
 $G=\phi(F)\cup \psi(F)\cup G_1$, apply Theorem \ref{decom} for $B_n,$ $G,$ $\varepsilon_1$ and $L$, and use Lemma \ref{rrpro}, to obtain a finite set $G'\subset B_n$ and $r\geq n$ such that $\nu: B_n\rightarrow B_r$ satisfies
$$
\nu(G')\subset_{1/6}\{f\in B_r\mid  f\,\text{has finite spectrum}\}.
$$
Then there exist a projection $p\in B_r$ and a unital homomorphism $\lambda:B_n\rightarrow (1-p)B_r(1-p)$ with finite dimensional image such that
 $$
 \|\nu(g)p-p\nu(g)\|<\varepsilon_1,
 $$
 $$
 \|\nu(g)-p\nu(g)p-\lambda(g)\|<\varepsilon_1
 $$
 for all $g\in G$, and
 $$
 [\lambda(1)]\geq L\cdot[p]\quad{\rm in}\quad K_0(B_r).
 $$

 Define a unital map $\theta:\,B_n\to pB_rp$ by $\theta(g)=p\nu(g)p$. It follows that $\theta$ is $\varepsilon_1$-multiplicative on $G_1$.
  Therefore, we can obtain $w_1$ and $\mu:B_n\rightarrow M_k(pB_rp)$ as in (b).
 It follows that
 \begin{equation}\label{}
 \|\nu\circ\phi(f)-\theta\circ\phi(f)-\lambda\circ\phi(f)\|<\varepsilon_1<\varepsilon,
 \end{equation}
 \begin{equation}\label{}
\|\nu\circ\psi(f)-\theta\circ\psi(f)-\lambda\circ\psi(f)\|<\varepsilon_1<\varepsilon
 \end{equation}
 for all $f\in F$.

  For the special case that $p=0$, then
  $$
  \nu\circ\phi\sim_{F,\,\varepsilon}\lambda\circ\phi,\quad
  \nu\circ\psi\sim_{F,\,\varepsilon}\lambda\circ\psi.
  $$
  Since $\mathrm{KK}(\lambda\circ\phi)=\mathrm{KK}(\lambda\circ\psi)$, by Lemma \ref{unires1}, it follows that
  $$
  \lambda\circ\phi\sim_{F,\,4\omega(F)}\lambda\circ\psi.
  $$
  Then we will have
 $$
 \nu\circ\phi\sim_{F,\,2\varepsilon+4\omega(F)}\nu\circ\psi.
 $$

 (d) If $p\neq 0$, the next step is to show that $[\lambda\circ\phi]-[\mu]$ is $m$-large. Since $[\mu(e)]\leq k\cdot[p]$, we have
 \begin{eqnarray*}
 % \nonumber % Remove numbering (before each equation)
   [\lambda\circ\phi(e)]-[\mu(e)] &=& [\lambda\circ\phi(e)]-[\mu(e)] \\
    &\geq & [\lambda(1)]-k\cdot[p] \\
    &\geq & rm\cdot[p].
 \end{eqnarray*}
 Then $[\lambda\circ\phi]-[\mu]$ is $m$-large.

(e) Since $\mathrm{KK}(\lambda\circ\phi)=\mathrm{KK}(\lambda\circ\psi)$ and $\lambda\circ\phi$, $\lambda\circ\psi$ have finite dimensional image, it follows from  Lemma \ref{simcase1}
 that there exists a unitary $u_1\in (1-p)B_r(1-p)$ such that
 \begin{equation}
 \|u_1\lambda\circ\phi(f)u_1^*-\lambda\circ\psi(f)\|<4\omega(F).
 \end{equation}
 Since $[\lambda\phi]-[\mu]$ is $m$-large, by Lemma \ref{simcase2}, there exists a homomorphism $\mu_0:B_n\rightarrow M_k(B_r)$ whose image is finite dimensional and orthogonal to the image of $\mu$, and there exists a unitary $u_2\in M_k(B_r)$ such that
 \begin{equation}
 \|u_2\lambda\circ\phi(f)u_2^*-\mu(f)-\mu_0(f)\|<4\omega(F)
 \end{equation}
 for all $f\in F$.

 (f) Now we have
 $$
 \nu\circ\phi\sim_{F,\,\varepsilon}\theta\circ\phi+\lambda\circ\phi\quad {\rm by\,(5.2)}
 $$

 $$
 \theta\circ\phi+\lambda\circ\phi\sim_{F,\,4\omega(F)} \theta\circ\phi+\mu+\mu_0\quad {\rm by\,(5.5)}
 $$

 $$
 \theta\circ\phi+\mu+\mu_0\sim_{F,\,\varepsilon} \theta\circ\psi+\mu+\mu_0\quad {\rm by\,(5.1)}
 $$

 $$
 \theta\circ\psi+\mu+\mu_0\sim_{F,\,4\omega(F)} \theta\circ\psi+\lambda\circ\phi\quad {\rm by\,(5.5)}
 $$

 $$
 \theta\circ\psi+\lambda\circ\phi\sim_{F,\,4\omega(F)}\theta\circ\psi+\lambda\circ\psi\quad {\rm by\,(5.4)}
 $$

 $$
 \theta\circ\psi+\lambda\circ\psi\sim_{F,\,\varepsilon}\nu\circ\psi\quad {\rm by\,(5.3)}
 $$
 Summing up, we have
 $$
 \nu\circ\phi\sim_ {F,\,3\varepsilon+12\omega(F)}\nu\circ\psi.
 $$

 Then there exists a unitary $u\in B_r$ such that
 $$
 \|u\nu_{n,r}\circ\phi(f)u^*-\nu_{n,r}\circ\psi(f)\|<3\varepsilon+12\omega(F)
 $$
for all $f\in F$.
\end{proof}

\textbf{Finally we can proceed to the proof of Theorem \ref{cla}:}
\begin{proof}Denote by $\alpha:\,\underline{\mathrm{K}}(A)\to\underline{\mathrm{K}}(B)$ the given isomorphism of ordered graded
$\Lambda$-modules. Let $\beta=\alpha^{-1}$, $\phi_n:\,A_n\to A$ and $\psi_n:\,B_n\to B$ be the obvious maps. Recall that by Theorem \ref{UCT}, we may identify ${\bf Hom}_\Lambda(\underline{\mathrm{K}}(A_i),\underline{\mathrm{K}}(B_j))$
with $\mathrm{KK}(A_i,B_j)$.

We construct a commutative diagram
$$
\xymatrixcolsep{3pc}
\xymatrix{
{\,\,\underline{\mathrm{K}}(A_{r_1})\,\,} \ar[d]_-{\rho_1} \ar[r]^-{\phi_{r_1,r_2*}}
& {\,\,\underline{\mathrm{K}}(A_{r_2})\,\,} \ar[d]_-{\rho_2} \ar[r]^-{}
& {\,\,\cdots\,\,} \ar[r]^-{}
& {\,\,\underline{\mathrm{K}}(A)\,\,}\ar[d]_-{\rho}
 \\
{\,\,\underline{\mathrm{K}}(B_{s_1})\,\,} \ar[r]_-{\psi_{s_1,s_2*}} \ar[ru]^-{\sigma_1}
& {\,\,\underline{\mathrm{K}}(B_{s_2}) \,\,} \ar[r]_-{} \ar[ru]^-{\sigma_2}
& {\,\,\cdots \,\,} \ar[r]_-{}
& {\,\,\underline{\mathrm{K}}(B)\,\,}}
$$
where $\rho_n,\,\sigma_n$ are liftable to $*$-homomorphisms $\xi_n:\,A_{r_n}\to B_{s_n}$
and $\psi_n:\,B_{s_n}\to A_{r_{n+1}}$. The construction is done inductively. We may assume that $A_{r_1}=B_{s_1}={0}$, and hence take $\rho_1=\sigma_1=0$.
Assume now that $\rho_i$ and $\sigma_i$ have been constructed for all $i\leq n-1$.
For $C^*$-algebra $A_{r_n}$, let $F\subset \underline{\mathrm{K}}^+(A_{r_n})$ be provided by Theorem \ref{exist}.
By Proposition \ref{dg finite}, the $\Lambda$-module $\underline{\mathrm{K}}(A_{r_n})$ is finitely generated,
there is $k\geq s_{n-1}$ and there is a $\xi\in{\bf Hom}_\Lambda(\underline{\mathrm{K}}(A_{r_n}),\underline{\mathrm{K}}(B_k))$ such that
$$
\psi_{k*}\xi=\rho\phi_{r_n *},\quad
\xi\sigma_{n-1}=\psi_{s_{n-1},k *},\quad
\xi(F)\subset\underline{\mathrm{K}}^+(B_k),
$$
and
$$
\xi[1_{A_{r_n}}]\leq[1_{B_k}].
$$
Then by Theorem \ref{exist},
$\xi$ can be lifted as a $*$-homomorphism.
We conclude the construction of $\rho_n$ by setting $k=s_n$ and $\rho_n=\xi$.
It is clear that $\rho_n\sigma_{n-1}=\psi_{s_n,s_{n+1}*}$, $\rho\phi_{r_n *}=\psi_{s_n *}\rho_n$.
Let $\xi_n:\,A_{r_n}\to B_{s_n}$ be a $*$-homomorphism implementing $\rho_n$.
The construction of $\sigma_n$ is similar.
This establishes a commutative diagram in the $\mathrm{KK}$-category. Then applying Theorem \ref{unithm0} and a standard intertwining argument, one can proceed in the same way as Theorem 7.3 of \cite{DG} to conclude an isomorphism of the algebras which implements the given isomorphism $\alpha$.
\end{proof}

\proof[Acknowledgements] The research of the second author is supported by a
grant from the Natural Sciences and Engineering Research
Council of Canada. He is indebted to The Fields Institute
for their generous support of the field of operator algebras. Part of this research was carried out during a visit
of the second author to Chongqing University; he thanks Chongqing University for its hospitality.

The third author is supported by the National Science Foundation of China with grant No.\,11501060; he is also supported by the Fundamental Research Funds for the Central Universities (Project No.\,2018CDXYST0024 in Chongqing University). The first and the fourth author are supported by the postdoctoral research station of mathematics in Hebei Normal University. Part of this research was carried out during visits of three of us to The Fields Institute; we thank The Fields Institute for its hospitality.

\end{document}